\documentclass[11pt]{article}
   \usepackage{amsfonts}
\usepackage{latexsym,tikz}
\usepackage{amsmath}
 
\title{Multiple Dedekind Zeta Functions}
\author{Ivan Horozov, Zhenbin Luo}
\date{October 23, 2013}

\newcommand \nc {\newcommand}
\nc \proof {\noindent {\em{Proof.\/ }}} \nc \qed {$\Box$\hfill}
\newtheorem{theorem}{Theorem}[section]
\newtheorem{lemma}[theorem]{Lemma}
\newtheorem{proposition}[theorem]{Proposition}
\newtheorem{corollary}[theorem]{Corollary}
\newtheorem{definition}[theorem]{Definition}
\newtheorem{example}[theorem]{Example}
\newtheorem{remark}[theorem]{Remark}
\newtheorem{conjecture}[theorem]{Conjecture}
\newtheorem{question}[theorem]{Question}
\nc \bth[1] {\begin{theorem}\label{t#1} } \nc \ble[1]
{\begin{lemma}\label{l#1} } \nc \bpr[1]
{\begin{proposition}\label{p#1} } \nc \bco[1]
{\begin{corollary}\label{c#1} } \nc \bde[1]
{\begin{definition}\label{d#1}\rm } \nc \bex[1]
{\begin{example}\label{e#1}\rm } \nc \bre[1]
{\begin{remark}\label{r#1}\rm } \nc \bcon[1]
{\begin{conjecture}\label{con#1}\rm } \nc \bque[1]
{\begin{question}\label{que#1}\rm }
\nc {\eth} { \end{theorem} } \nc {\ele} { \end{lemma} } \nc
{\epr}{ \end{proposition} } \nc {\eco} { \end{corollary} } \nc
{\ede} {\end{definition} } \nc {\eex} { \end{example} } \nc {\ere}
{\end{remark} } \nc {\econ} { \end{conjecture} } \nc {\eque}
{\end{question} }

\newcommand{\gte}{\grave{\text{e}}}
\newcommand{\ccs}{\text{Contou-Carr$\gte$re symbol\ }}

\def\ga{\gamma}

\def\e{\epsilon}

\def\om{\omega}

\def\sg{\sigma}

\def \P {{\mathbb P}}

\def \Z {{\mathbb Z}}

\def \C {{\mathbb C}}

\begin{document}

\title{{\LARGE\bf{
On the Contou-Carrere Symbol for Surfaces} 
}}

\maketitle

\begin{abstract}
This is a preliminary report on the Contou-Carr\'ere symbol for surfaces. It consists of two parts. In the first part, we recall technical results needed to define the symbol. The second part is where we compute all components of the Coutou-Carr\'ere symbol for surfaces, using iterated integrals over membranes.
\end{abstract}



\tableofcontents
\setcounter{section}{-1}

\section{Introduction}
In this paper we give a definition of a two dimensional Contou-Carrere symbol. Previously, Pablos Romo has worked on that topic. He gave a definition whose ingredients resemble powers of the one dimensional Contou-Carrere symbol. Then Osipov and Zhu gave another definition, in which they considered one more key ingredient of the Countou-Carrere symbol. Here, we find one more ingredient  for the Contou-Carrere symbol (Case 2, page 15 or $Q_1$ and $Q_2$ in Definition 2.2).

We can consider the results here as a continuation of \cite{rec2}. In order to construct the Contou-Carrere symbol, we use iterated integrals over membranes, which are a higher dimensional generalization of iterated path integrals.  The geometric constructions that we use here are established in the paper  \cite{rec2}. They allow us to prove reciprocity laws for the symbols. However, in this preliminary version, we only state the Contou-Carrere symbol.

In the last Subsection we compute an example of a symbol from more complicated factors.
We are going to prove reciprocity laws for symbols coming from more complicated factors such as rational functions over $\P^2.$

Let $f$ and $g$ be two non-zero meromorphic functions on a Riemann surface $\Sigma$. The tame symbol of $f$ and $g$ at a point $x\in \Sigma$ is given by
$$(f,g)_x = (-1)^{\nu(f) \nu(g)} [g^{\nu(f)}/f^{\nu(g)}](x)$$
where $\nu(f)$ is the valuation of $f$ at $x$, in other words, it is the order of zero of $f$ at $x$, or the opposite of the order of pole of $f$ at $x$. Since the valuation of $g^{\nu(f)}/f^{\nu(g)}$ is zero, the quotient function is holomorphic at $x$ and the symbol is well defined. This symbol is anti-symmetric and, although not obvious, bi-multiplicative and satisfies the Steinberg property:

$$(f,1-f)_x = 1$$
for any $f$ not equal to 0 or 1. When the surface $\Sigma$ is compact,  this symbol satisfies the following product formula:

$$\prod_{x \in \Sigma} ( f, g)_x = 1$$
which generalizes the Weil reciprocity law \cite{Mi},  \cite{W} $f(\text{div}(g)) = g(\text{div}(f))$, for the case when $f$ and $g$ have disjoint divisors.

 This symbol is later greatly generalized by \ccs \cite{Co} to the case where the base ring for the surface is not the complex number $\mathbb{C}$ but a local artinian ring. Let $A$ be an artinian local ring with maximal ideal $m$. Let $f, g$ be two functions in $A((X))^{\times}. $ Then they can be written uniquely as
$$f=a_0 X^{\nu(f)} \prod_{i=-\infty}^{+\infty} (1-a_i X^i)$$
$$g=b_0 X^{\nu(g)} \prod_{i=-\infty}^{+\infty} (1-b_i X^i)$$
where $(f,g)$ is the greatest divisor of $f$ and $g$, $\nu_f, \nu_g \in \Z$, $a_i, b_i \in A$ for $i>0$, $a_0, b_0 \in A^{\times}$, $a_i, b_i \in m$ for $i<0$, and $a_i, b_i$ are zero when $i \ll 0$.
Then the symbol is given by
\begin{eqnarray}
<f,g>_{A((X))^{\times}} :=\\
:= (-1)^{\nu_f \nu_g} \dfrac{ a_0^{\nu_g} \prod_{j=1}^{\infty}\prod_{k=1}^{\infty} (1 - a_{j}^{k /(j,k)}b_{-k}^{j/(j,k)})^{(j,k)}}{b_0^{\nu_f} \prod_{j=1}^{\infty}\prod_{k=1}^{\infty} (1 - a_{-j}^{k /(j,k)}b_k^{j/(j,k)})^{(j,k)}}
\end{eqnarray}
Since $a_{-i}, b_{-i}$ are zero when $i$ is large, the product is actually finite, hence the definition makes sense.

There is also an analogous reciprocity law satisfied by this Contou-Carr\'ere symbol. Let $A$ be a finite local artinian $\mathbb{C}$-algebra, finitely generated as a $\mathbb{C}$-module, and let  $\Sigma'$ be a surface
over $A$. Let $f$ and $g$ be two non-zero meromorphic functions on $\Sigma'$. Then locally at any point $x$, $f$ and $g$ can be identified as elements of the ring $A((x))^\times$. Then we can define the \ccs described above. This symbol satisfies the product formula

$$\prod_{x \in \Sigma'} <f,g>_x = 1.$$

If we take $A=\mathbb{C}$, then the \ccs reduces to the tame symbol and the above product formula becomes the Weil reciprocity formula. Hence the \ccs is a natural extension of the tame symbol.

In the paper \cite{Luo}, the second author constructs the \ccs in terms of Chen iterated integrals and prove the corresponding reciprocity using a geometric property of iterated integrals. Consider a local artinian $\C$-algebra $A$, finitely generated by nilpotent elements $\{a_1, a_2, \cdots, a_n\}$ as a $\C$-module. Let $\om$ be a meromorphic 1-form on a compact Riemann surface $X$ with coefficients in $A$, in other words, it is of the form $\sum_{i=0}^n a_i \om_i$, where $a_0=1$ and $\om_i$'s are regular meromorphic 1-forms. Let $\ga$ be a path on the Riemann surface. Define the integral of $f$ by extending the regular integral linearly:
$$\int_\ga \om = \sum_{i=0}^n a_i \int_\ga \om_i.$$
In particular, we consider differential forms of the type $\frac{df}{f}$, where $f$ is a meromorphic function on $X$ with coefficients in $A$. Let $f$ and $g$ be two such functions, and $s$ be a zero or pole of $f$ or $g$ or both. Write $f$ and $g$ in Laurent series in
powers of a uniformizer $x_s$, as elements in $A((x_s))^{\times}$. Thus, as mentioned above, we can express $f$ and $g$ locally as infinite products. Let $\sigma$ be a simple loop that starts and
ends at some fixed point $P$, go around the divisor $s$ once in the counterclockwise
direction but not any other divisors of $f$ or $g$. Then we have the
following result:

\begin{eqnarray}
\exp\left(\frac{1}{2\pi i}\int_{\sg} \frac{df}{f} \circ \frac{dg}{g}\right) 
=\\
=
 (-1)^{\nu_f \nu_g} \frac{g(P)^{\nu_f} a_0^{\nu_g} \prod_{j=1}^{\infty}\prod_{k=1}^{\infty} (1 - a_{j}^{k /(j,k)}b_{-k}^{j/(j,k)})^{(j,k)}}{f(P)^{\nu_g}b_0^{\nu_f} \prod_{j=1}^{\infty}\prod_{k=1}^{\infty} (1 - a_{-j}^{k /(j,k)}b_k^{j/(j,k)})^{(j,k)}}
\end{eqnarray}

This formula differs from the \ccs symbol only by $g(P)^{\nu_f}/f(P)^{\nu_g}$, but since the point $P$ is fixed, we can rescale the functions $f$ and $g$ so that $f(P)=g(P)=1$. Then we can get the usual \ccs. This gives us a new interpretation of the $\ccs$ in terms of exponential
of an iterated integral, which is very convenient to use. Many properties of the symbol easily follow from the properties of iterated integrals. For example, it is easy to see from this definition that the symbol is bi-multiplicative and anti-symmetric. It is actually also convenient to prove the Steinberg property using this new description of the \ccs.

 We can also reproduce the product formula. If we consider the iterated integral over a loop that represents a relation of the fundamental group, by the homotopy invariance of iterated integral, the result is 0. Using this, we can reprove the reciprocity of \ccs easily.

In the two dimensional setting we do have a similar infinite product formulas. We have 
\[f_1=a_1x^{\nu_1(f_1)}y^{\nu_2(f_1)}\prod_{i_1>-N}\prod_{j_1>-N_{i_1}}(1-a_{i_1,j_1}x^{i_1}y^{j_1}),\]
\[f_2=a_2x^{\nu_1(f_2)}y^{\nu_2(f_2)}\prod_{i_2>-N}\prod_{j_2>-N_{i_2}}(1-a_{i_2,j_2}x^{i_2}y^{j_2}),\]
\[f_3=a_3x^{\nu_1(f_3)}y^{\nu_2(f_3)}\prod_{i_3>-N}\prod_{j_3>-N_{i_3}}(1-a_{i_3,j_3}x^{i_3}y^{j_3}).\]
For each triple of simple factors of the  above infinite products, we compute the corresponding symbol in Section 2. We have separated the computation into 8 cases. At the end of the computation we define the \ccs in Definition 2.2.

\section{Algebraic, Geometric and Analytic background}
\subsection{Infinite product formulas}

\smallskip
   Let $A$ be a commutative ring with unit. Let $I$ be a nilpotent
   ideal. Let $\Gamma(A,I)$ be the set of power series $f = \sum_{i=-\infty}^{\infty}a_i t^i \in
   A((t))$ such that for some integer $\omega = \omega(f) =
   \omega_{A,I}(f)$ we have $a_w\in A^{\times}$ and $a_i\in I$ for
   $i<\omega$. The set $\Gamma(A,I)$ is closed under power series
   multiplication and forms a group. Let $\Gamma_0(A,I)=\{ f\in\Gamma(A,I)\mid
   \omega(f)=0\}$. It is a subgroup of $\Gamma(A,I)$. Let
   $\Gamma_-(A,I)$ be the subgroup consisting of the elements of the
   form $1+f$ where $f \in t^{-1}I[t^{-1}]$. Let $\Gamma_+(A,I)$ be
   the subgroup of $\Gamma_0(A,I)$ consisting elements which have
   no negative powers of t. Then $\Gamma_+(A,I) = A[[t]]^{\times}$.
   Given rings $A$ and $B$ with nilpotent ideal $I\subset A$ and $J\subset
   B$, and a ring homomorphism $\phi: A \rightarrow B$ such that $\phi (I) \subset
   J$, then if we define the corresponding group homomorphism $\Gamma(\phi): \Gamma(A,I) \rightarrow
   \Gamma(B,J)$ by sending $\sum_i a_i t^i$ to $\sum_i
   \phi(a_i)t^i$, then the construction $\Gamma$ becomes a functor.
   Similarly, $\Gamma_\pm$ and $\Gamma_0$ are functors.

   The following proposition is proved by Pablos Romo\cite{R2}.

\begin{proposition}
  For all $f \in \Gamma(A,I)$, there exist unique coefficients
$\{a_i\}_{i=-\infty}^\infty$ in $A$ satisfying $a_0 \in A, a_i \in I$
for $i < 0, a_i = 0$ for $i \leq 0$, such that $$f = a_0 t^{w(f)}
\displaystyle\prod_{i=1}^\infty(1-a_i t^i)
\displaystyle\prod_{i=1}^\infty(1-a_{-i}t^{-i})$$. 
\end{proposition}

\bigskip
The coefficients $\{a_i\}_{i=-\infty}^\infty$ are called the
family of \emph{Witt parameters} of $f \in \Gamma(A,I)$

The following two lemmas are used to prove the proposition.

\begin{lemma}
 For all $f \in \Gamma(A,I)$, there exists unique  $g \in \Gamma_+(A,I), h
\in \Gamma_-(A,I)$, such that $f = g\cdot h$.
 \end{lemma}

\smallskip 
\begin{lemma}
 For all $f \in \Gamma_+(A,I)$, $f$ can be written as
$f(0)\prod_{i=1}^{\infty}(1-a_i t^i)$ for some 
$\{a_i\}_{i=1}^\infty$ uniquely determined by $f$. If $f \in 1 + t^n
A[[t]]^{\times}$, then $a_1=a_2=\cdots=a_{n-1}=0$. If $f \in 1 +
I[t]$, then $a_i \in I, \forall \ i,$ and $a_i=0$, for $i\ll 0$.
\end{lemma}

See Anderson and Pablo Romo (\cite{A-P}.

\bigskip

\smallskip

There is a similar result for higher dimensional local fields: Let
$A$ again be a ring, $I \subset A$ be a nilpotent ideal. Consider
the two dimensional local field $A((t_1))((t_2))$. Let
$\Gamma(A((t_1)),I((t_1)))$ be the subset of $A((t_1))((t_2))$
that consist of all the elements of the form $$\displaystyle\sum_{i=-\infty}^{\infty}g_{i}(t_1){t_2}^i =
\sum_{i=-\infty}^{\infty}  (\sum_{j=-\infty}^{\infty} a_{ij}
{t_1}^{j}) {t_2}^{i}$$ such that there exists $
\omega_2=\omega_2(f)\in \Z,\  \omega_1=\omega_1(f) \in \Z \mbox{ satisfying: } g_{\omega_2}(t_1)\in A((t_1))^{\times},\ g_i(t_1) \in
I((t_1)), \forall \ i<\omega_2 \mbox{ and } a_{\omega_2\omega_1}\in
A^{\times},\ a_{\omega_2 j}\in I, \forall j<\omega_1 .$

\smallskip
\begin{proposition}
 (case of dimension 2) For all $f \in \Gamma (A((t_1)),I((t_1)))$,
there exists a unique family of parameters
$\{a_{i,j}\}_{i,j=-\infty}^\infty$ such that $$f = a_{0,0}
t_1^{\omega_1} t_2^{\omega_2}\displaystyle\prod_{i=-\infty}^\infty
\prod_{j=- \infty}^{\infty}(1-a_{i,j}t_1^j t_2^i)$$. 
\end{proposition}

To prove it, we'll need the following lemmas:

\begin{lemma}
 For all $f \in 1+t_2A((t_1))[[t_2]]$, there exists $a_{ij} \in A$ uniquely determined by $f$, such that $$f = \displaystyle\prod_{i=1}^{+\infty}
\prod_{j=-N_i}^{+\infty}(1-a_{ij}t_1^j t_2^i)$$. 
\end{lemma}

\begin{proof}

 Suppose $$f_1(t_1) = \sum_{j=-N_1}^{+\infty} a_{1,j}
t_1^j.$$ Let $$A_1 = \prod_{j=-N_1}^{+\infty} (1+a_{1,j}t_1^j
t_2).$$ Then $A_1$ is invertible in $A((t_1))((t_2))$. Define
$f^{(1)}:= A_1^{-1} f $, then $f = A_1 f^{(1)}$ and $$f^{(1)}= 1 +
\sum_{i=2}^{+\infty} f^{(1)}_i(t_1) t_2^i .$$ Suppose at
$n^{\text{th}} $ step we get $f = (\prod_{i=1}^{n}A_i)(f^{(n)}) $,
where: $$A_i = \displaystyle\prod_{j=-N_i}^{+\infty} (1+a_{i,j}t_1^j
t_2^i ),\ f^{(n)} = 1+\displaystyle\sum_{i=n+1}^{+\infty}
f_i^{(n)}(t_1) t_2^i.$$ If $$f_{n+1}^{(n)} =
\displaystyle\sum_{j=-N_{n+1}}^{+\infty} a^{(n)}_{n+1,j} t_1^j$$
then we can similarly set $$A_{n+1}=
\displaystyle\prod_{j=-N_{n+1}}^{+\infty}(1+a^{(n)}_{n+1,j}t_1^j
t_2^i).$$ Still this is invertible in $A((t_1))((t_2))$ and we can
set $$f^{(n+1)}:= A_{n+1}^{-1} f^{(n)} =
1+\displaystyle\sum_{i=n+2}^{+\infty} f^{(n+1)}_i (t_1) t_2^i .$$ So
$$f = \displaystyle\prod_{i=1}^{n+1} A_i f^{(n+1)}.$$ Continue this
process, we'll get that $$f = \displaystyle\prod_{i=1}^{+\infty}
\prod_{j=-N_i}^{+\infty}(1-a_{ij}t_1^j t_2^i).$$ Obviously $a_{ij}
\in A$ are uniquely determined by $f$ according to this process.
\end{proof}

\begin{lemma}
 For all $f \ \in 1-t_2^{-1}I((t_1))[t_2^{-1}]$, then there exist $N, N_i \ \in
\Z$, and $a_{ij} \in I$ uniquely determined by $f$, such that
$$f = \displaystyle\prod_{i=-N}^{-1}
\prod_{j=-N_i}^{+\infty}(1-a_{ij}t_1^j t_2^i).$$ 
\end{lemma}

\begin{proof}
Induction on the integer $k$ such that $I^k = 0$. If $k = 1$, then
$f = 1$, so the statement is true. Now suppose the statement is true
for all the ideals $J$ satisfying $J^i=0$ for some $i<k$ and $I^k=0,
I^{k-1} \neq 0$. Suppose $$f = 1-\displaystyle\sum_{i=-n}^{-1}f_i
(t_1) t_2^i, $$ where $f_i (t_1) \ \in I((t_1))$. Write
$f_{-1}(t_1)$ as $$\displaystyle\sum_{j=-N_1}^{+\infty} a_{1,j}
t_1^j,$$ where $\ a_{1,j} \ \in I$. Then we can similarly set
$$A_{-1} = \displaystyle\prod_{j=-N_1}^{+\infty} (1-a_{1,j}t_1^j
t_2^{-1}).$$ Then $A_{-1}$ is invertible and its inverse is
$$\displaystyle\prod_{j=-N_1}^{+\infty}
(1+\displaystyle\sum_{l=1}^{+\infty} (a_{1,j}t_1^j t_2^{-1})^l).$$
Since $I$ is nilpotent, there exists an integer $k$, such that $I^k
= 0$. So every factor of the inverse has only finite terms, ie.
$$(A_{-1})^{-1} = \displaystyle\prod_{j=-N_1}^{+\infty}
(1+\displaystyle\sum_{l=1}^{k-1} (a_{1,j}t_1^j t_2^{-1})^l).$$ Let
$f^{-1} = (A_{-1})^{-1} f$. Then $f^{(-1)} \ \in 1 - t_2^{-2}
I((t_1))[t_2^{-1}]$ and the lowest power of $t_2$ in $f^{(-1)}$
would be $-n-k+1$. Write $$f^{(-1)} = 1-
\displaystyle\sum_{i=-2}^{-n-k+1} f^{(-1)}_{i}(t_1) t_2^i,$$ then
for $-n \leq i \leq -2, f^{(-1)}_{i} \  \in I((t_1))$, and for
$i<-n, f_{i}^{(-1)} \ \in I^{-i-n-1}((t_1))$. As before, if we write
$f^{(-1)}_{-2}$ as $$\sum_{j=-N2}^{+\infty} a^{(-1)}_{-2, j},$$ and
set $$A_{-2} = \prod_{j=-N2}^{+\infty} (1- a^{-1}_{-2,j} t_1^{j}
t_2^{-2} ), f^{(-2)} = (A_{-2}){-1} f = 1+ \sum_{i = -3}^{ -n-k+1}
f^{-2}_i (t_1)t_2^i,$$ then for $-n \leq i \leq -3,\  f^{-2}_i \in
I((t_1))$, for $-n-k+1 < i < -n, \ f^{-2}_{i} \in
I^{-i-n-1}((t_1))$. If we go on with this process, after $n$ steps, we
will get $$f = \prod_{i=-n}^{-1} A_{-i} f^{(-n)},$$ where $f^{(-n)}
\ \in 1 - t_2^{-(n+1)}I^2((t_1))[t_2^{-2}]$. So by induction
hypothesis, $f^{(-n)}$ can be written as
$$\displaystyle\prod_{i=-N'}^{-1}
\prod_{j=-N'_i}^{+\infty}(1-a'_{ij}t_1^j t_2^i).$$ So $f$ can be
written as $$\displaystyle\prod_{i=-N}^{-1}
\prod_{j=-N_i}^{+\infty}(1-a_{ij}t_1^j t_2^i),$$ and $a_{ij} \in I$
are uniquely determined by $f$.
\end{proof}

Combining these two lemmas and lemma 1.2, we can get proposition 1.4
easily.

Using similar techniques, we can proof the following proposition for
any higher dimensional cases:
\smallskip
\begin{proposition}(Case of dimension $n$):

For any function $f \in \Gamma(A((t_1))\cdots ((t_n)),I(t_1))\cdots
((t_n)))$(which is defined similarly), there exist coefficients
$\{a_{i_1,i_2,\cdots,i_n}\}$ uniquely determined by $f$, such that
$$f = \prod_{i_n=-N}^{+\infty}\prod_{i_{n-1}=-N_{i_n}}^{+\infty}\cdots
\prod_{i_1=-N_{i_2}}^{+\infty} (1 -
a_{i_1,i_2,\cdots,i_n}t_1^{i_1}t_2^{i_2}\cdots t_n^{i_n}).$$
\end{proposition}

\subsection{Two foliations}
\label{subsec 2 foliations}

The goal of this section is to construct two foliations on a
complex projective algebraic surface $X$ in $\P^k$. Let $f_1$, $f_2$, $f_3$
and $f_4$ be four non-zero rational functions on the surface $X$.
Let
$$C\cup C_1\cup\dots\cup C_n=\bigcup_{i=1}^4|div(f_i)|.$$
Let $$\{P_1,\dots,P_N\}=C\cap (C_1\cup\dots\cup C_n).$$ We can
assume that the curves $C,C_1,\dots,C_n$ are smooth and that the
intersections are transversal (normal crossings), by allowing blow-ups on the surface $X$.

The two foliations have to satisfy the following
\\

\noindent
{\bf{Conditions:}}
\begin{enumerate}
\item There exists a foliation $F'_v$  such that
  \begin{enumerate}
    \item $F'_v=(f-v)_0$ are the level sets of a rational function $$f:X\rightarrow \P^1,$$ for small values of $v$, (that is, for $|v|<\e$ for a chosen $\e$);
    \item $F'_v$ is smooth for all but finitely many values of $v$;
    \item $F'_v$ has only nodal singularities;
    \item $ord_C(f)=1$;
    \item $R_i\notin C_j,$ for $i=1,\dots,M$ and $j=1,\dots,n$, where $$\{R_1,\dots,R_M\}=C\cap (D_1\cup\dots\cup D_m)$$ and $$F'_0=(f)_0=C\cup D_1\cup\dots\cup D_m.$$
  \end{enumerate}
\item There exists a foliation $G_w$ such that
  \begin{enumerate}
    \item $G_w=(g-w)_0$ are the level sets of a rational function $$g:X\rightarrow \P^1;$$
    \item $G_w$ is smooth for all but finitely many values of $w$;
    \item $G_w$ has only nodal singularities;
    \item $g|_C$ is non constant.
\end{enumerate}
\item Coherence between the two foliations $F'$ and $G$:
  \begin{enumerate}
    \item All but finitely many leaves of the foliation $G$ are transversal to the curve $C$.
    \item $G_{g(P_i)}$ intersects the curve $C$ transversally, for $i=1,\dots,N$. (For definition of the points $P_i$ see the beginning of this Subsection.)
    \item $G_{g(R_i)}$ intersects the curve $C$ transversally, for $i=1,\dots,M$. (For definition of the points $R_i$
see condition 1(e).)
\end{enumerate}
\end{enumerate}

For the existence of such foliations see \cite{rec2}.

\begin{lemma}
With the above notation, for small values of $|v|$, we have that $F_v$  has the homotopy type of $C_0$.
\end{lemma}
A proof can be found  in \cite{rec2}.

\subsection{Iterated integrals}

For proofs of theorems of this section, see Chen[Ch] or Goncharov[G].

\begin{definition}
\smallskip Let $\om_1, \om_2, \cdots, \om_n$ be holomorphic 1-forms
on a simply connected open subset $U$ of the complex plane $\C$. Let
$\ga : [0,1]\to U$ be a path. Then we call the integral
$$
\int_{\ga} \om_1 \circ  \cdots \circ \om_n := \int \cdots
\int_{0\leq t_1\leq\cdots\leq t_n\leq 1} \ga^* \om_1(t_1)\wedge
\cdots \wedge \ga^*\om_n(t_n)
$$
the iterated integral of the differential forms $\om_1, \om_2,
\cdots, \om_n$ over the path $\ga$.
\end{definition}

\begin{theorem}\label{hominv}
Let $\om_1, \cdots, \om_n$ be holomorphic 1-forms on a simply
connected open subset $U$ of the complex plane $\C$. Let $H : [0, 1]
\times [0, 1] \to U$ be a homotopy, fixing the end points, of paths
$\ga_s : [0, 1] \to U $such that $\ga_s(t) = H(s, t)$. Then
$$\int_{\ga_s} \om_1 \circ \cdots \circ \om_n $$ is independent of $s$.
\end{theorem}

\begin{theorem}\label{shurel}[Shuffle relation]
Let $\om_1, \cdots, \om_n, \om_{n+1}, \cdots, \om_{n+n}$ be
differential 1-forms, where some of them could repeat. Let also $\ga$ be a
path that does not pass through any of the poles of the given
differential forms. Denote by $Sh(m, n)$ the shuffles, which are
permutations $\tau$ of the set $\{1, . . . , m,m + 1, . . . ,m +
n\}$ such that $\tau(1) < \tau(2) < \cdots < \tau(m)$ and $\tau(m +
1) < \tau(m + 2) < \cdots < \tau(m + n)$. Then
$$ \int_{\ga} \om_1 \circ \cdots \circ \om_n \int_{\ga} \om_{n+1}
\circ \cdots \circ \om_{m+n} = \sum_{\tau \in Sh(m,n)} \int_{\ga}
\om_{\tau(1)} \circ \om_{\tau(2)} \cdots \circ \om_{\tau(m+n)}.$$
\end{theorem}

\begin{lemma}[Reversing the path]
$$\int_{\ga} \om_1 \circ \om_{2} \circ \cdots \circ \om_n = (-1)^n \int_{\ga^{-1}} \om_n \circ \om_{n-1} \circ \cdots \circ \om_1$$
\end{lemma}

\begin{theorem}\label{lemma composition of paths}
[Composition of paths]
Let $\om_1, \om_2, \cdots, \om_n$ be differential forms, where some of
them could repeat. Let $\ga_1$ be a path that ends at $Q$ and
$\ga_2$ be a path that starts at $Q$. Then
$$\int_{\ga_1\ga_2} \om_1 \circ \om_{2} \circ \cdots \circ \om_n = \sum_{i=0}^{n}\int_{\ga_1} \om_1 \circ \om_{2} \circ \cdots \circ \om_i \int_{\ga_2} \om_i \circ \om_{i+1} \circ \cdots \circ \om_n$$
\end{theorem}

Let $\tau$ be a simple loop around $C$ in $X-D$, based at $R$.
Let $\sigma$ be a loop on the curve $C^0=C_0-(D_1\cup\dots\cup D_m)\cap C_0$.
We define a {\it{membrane}} $m_{\sigma}$  associated to a loop $\sigma$ in $C^0$ by
$$m_{\sigma}:[0,1]^2\rightarrow X,$$
$$m_{\sigma}(s,t)\in F_{f(\tau(t))}\cap G_{g(\sigma(s))}$$
$$m_{\sigma}(0,0)=R.$$
Note that for fixed values of $s$ and $t$, we have that
$$F_{f(\tau(t))}\cap G_{g(\sigma(s))}$$
consists of finitely many points, where $F$ and $G$ are foliations
satisfying the Conditions in Subsection 2.1 and Lemma 2.2.

Consider the dependence of $\log(f_i(m(s,t))$ on the variables $s$
and $t$ via the parametrization of the membrane $m$. We have
$$d\log(f_i(m(s,t))=
\frac{\partial \log(f_i(m(s,t)))}{\partial s}ds+\frac{\partial \log(f_i(m(s,t)))}{\partial t}dt.
$$

In order to use a more compact notation, we will use
$$\log(f_i),_s(s,t)=\frac{\partial \log(f_i(m(s,t)))}{\partial s}$$
and similarly
$$\log(f_i),_t(s,t)=\frac{\partial \log(f_i(m(s,t)))}{\partial t}.$$

\begin{definition}(Interior Iterated integrals on membranes)
\label{def I} Let $f_1,\dots, f_{k+l}$ be rational functions on
$X$, where the integers $(k,l)$ will be superscripts. Let $m$ be
a membrane as above. We define:\\
\\
\noindent (a) $I^{(1,1)}(m;f_1,f_2)=$\\
\begin{eqnarray*}
\int_0^1\int_0^1\log(f_1),_s(s,t)ds\wedge \log(f_1),_t(s,t)dt;
\end{eqnarray*}

\noindent (b) $I^{(1,2)}(m;f_1,f_2,f_3)=$\\
\begin{eqnarray*}
=\int\int\int_{0\leq s_1\leq s_2\leq 1;0\leq t\leq1}
\log(f_1),_{s_1}(s_1,t)ds_1
\wedge
\log(f_2),_{t}(s_1,t)dt\wedge \\
\wedge
\log(f_3),_{s_2}(s_2,t)ds_2;
\end{eqnarray*}

\noindent (c) $I^{(2,1)}(m;f_1,f_2,f_3)=$
\begin{eqnarray*}
=\int\int\int_{0\leq s\leq 1;0\leq t_1\leq t_2\leq1}
\log(f_1),_{s}(s,t_1)ds
\wedge
\log(f_2),_{t_1}(s,t_1)dt_1\wedge\\
\wedge
\log(f_3),_{t_2}(s,t_2)dt_2;
\end{eqnarray*}
\end{definition}

Define any smooth metric on $X$. Let $\tau$ be a simple loop
around the curve $C$ of distance at most $\e$ from $C$. We are
going to take the limit as $\e\rightarrow 0$. Informally, the
radius of the loop $\tau$ goes to zero. Then we have the following
lemma.





Using Chen \cite{Ch} we obtain the following Lemma.

\begin{lemma}
\label{le commutator}
Let $\alpha$ and $\beta$ be two loops on the surface $X$ with a common base. 
Put
$$\theta_1=\frac{df_1}{f_1}$$
and 
$$\theta_2=\frac{df_2}{f_2}\wedge\frac{df_3}{f_3}.$$
Put $[\alpha,\beta]=\alpha\beta\alpha^{-1}\beta^{-1}.$
Then
$$\int_{[\alpha,\beta]}\theta_1\cdot\theta_2=
\int_{\alpha}\theta_1\int_\beta\theta_2
-
\int_{\beta}\theta_1\int_\alpha\theta_2.
$$
\end{lemma}
\proof It follows directly from Lemma \ref{lemma composition of paths} by applying it to each ingredient of the commutator.

As a direct consequence, we obtain the following:
\begin{corollary}
\label{prop commutator}
$$\int_{m_{[\alpha,\beta]}}
\frac{df_1}{f_1}\cdot\left(\frac{df_2}{f_2}\wedge\frac{df_3}{f_3}\right)
\in 
(2\pi i)^3 \Z.
$$
\end{corollary}

Following Chen, we obtain the $1$-form 
\[\int\frac{df_1}{f_1}\circ\left(\frac{df_2}{f_2}\wedge\frac{df_3}{f_3}\right)\]
on the loop space is closed since

(1) $\frac{df_1}{f_1}$ and $\frac{df_2}{f_2}\wedge\frac{df_3}{f_3}$ are closed and

(2)  $\frac{df_1}{f_1}\wedge\frac{df_2}{f_2}\wedge\frac{df_3}{f_3}=0.$

Since, we have a closed form it follows that  the integral is homotopy invariant. 
Thus, we can take a relation in the fundamental group of a curve embedded in the surface. More precisely, we take
\[\delta=[\alpha_1,\beta_1]\dots[\alpha_g,\beta_g]\sigma_1\dots\sigma_n\]
for a curve $C_0$ of genus $g$ with $n$ punctures.

For each of the above loops we associate a torus.

Let $\tau$ be a simple loop around $C_0$ in $X-C_0-\left(\bigcup_{i=1}^M G_{g(U^\e_i)}\right)$.
Let $\sigma$ be a loop on the curve $C_0$.
We define a {\it{membrane}} $m_{\sigma}$  associated to a loop $\sigma$ in $C^0$ by
$$m_{\sigma}:[0,1]^2\rightarrow X$$
and 
$$m_{\sigma}(s,t)\in F_{f(\tau(t))}\cap G_{g(\sigma(s))}.$$
Note that for fixed values of $s$ and $t$, we have that
$$F_{f(\tau(t))}\cap G_{g(\sigma(s))}$$
consists of finitely many points, where $F$ and $G$ are foliations
satisfying the Conditions in Subsection \ref{subsec 2 foliations}.

{\bf{Claim:}} The image of $m_\sigma$ is a torus.

Indeed, consider a tubular neighborhood around a loop $\sigma$ on the curve $C_0$. One can take the following tubular neighborhood: 
\[\bigcup_{|v|<\e}F_v\cap G_{g(\sigma)}\]
of $F_v\cap G_{g(\sigma)}$. Its boundary is $F_{f(\tau)}\cap G_{g(\sigma)}$, where $\tau$ is a simple loop around $C_0$ on $X-\bigcup_{i=1}^n C_i -\bigcup_{j=1}^m D_j$ and $|f(\tau(t))|=\e$.

In the last section we will associate a Contou-Carrere symbol to a simple loop $\sigma_i$, 
namely,  $I^{1,2}_{m_{\sigma_i}}(f_1,f_2,f_3)$. By the above corollary we have that 
$I^{1,2}_{m_{[\alpha_i,\beta_i]}}(f_1,f_2,f_3)$ is an integer multiple of $(2pi i)^3$

Then $I^{1,2}_{m_\delta}=0$, since $\delta$ is homotopic to the trivial path.
Also, by the  Lemma \ref{lemma composition of paths}, we have
$0=I^{1,2}_{m_\delta}(f_1,f_2,f_3)=\sum_{i=1}^nI^{1,2}_{m_{{\sigma_i}}}(f_1,f_2,f_3) +(2\pi i)^3\Z$,

\section{Countou-Carrere symbol for surfaces and its reciprocity laws}

\subsection{Cocycle on the loop space of a surface}

\subsection{Semi-local symbol}

In this subsection we present computation of the Contou-Carrere symbol
for all possible factors from the formal infinite product. 
Let 
\[f_1=x^{\nu_1(f_1)}y^{\nu_2(f_1)}\prod_{i_1>-N}\prod_{j_1>-N_{i_1}}(1-a_{i_1,j_1}x^{i_1}y^{j_1}),\]
\[f_2=x^{\nu_1(f_2)}y^{\nu_2(f_2)}\prod_{i_2>-N}\prod_{j_2>-N_{i_2}}(1-a_{i_2,j_2}x^{i_2}y^{j_2}),\]
\[f_3=x^{\nu_1(f_3)}y^{\nu_2(f_3)}\prod_{i_3>-N}\prod_{j_3>-N_{i_3}}(1-a_{i_3,j_3}x^{i_3}y^{j_3}).\]
We consider an integral over a torus of
$\log(f_1)\frac{df_2}{f_2}\wedge\frac{df_3}{f_3}$ in the following cases:
The function $f_1$ is either $x^{i_1}y^{j_1}$ or $1-x^{i_1}y^{j_1}$. 
The function $f_2$ is either $x^{i_2}y^{j_2}$ or $1-x^{i_2}y^{j_2}$.  
 The function $f_3$ is either $x^{i_3}y^{j_3}$ or $1-x^{i_3}y^{j_3}$. 
For each of the functions there are two possibilities. For the triple $(f_1,f_2,f_3)$ there are $2^3$ possibilities, which we list in the following $8=2^3$ cases.
We define the two dimensional Contou-Carrere symbol
as a cyclic symmetrization of  
\[\exp\left(\int\int_{T}\log(f_1)\frac{df_2}{f_2}\wedge\frac{df_3}{f_3}\right),\]
where $T$ is a torus of the type  $m_\sigma$ from Section 1.
At the next 8 cases, we examine the logarithm of the Countou-Carrere symbol,
when each of the functions $f_1,f_2,f_3$ consists of a single factor of the above infinite products.

At the end of the paper, we compute a more complicated case, which will be useful when we consider complex analytic products instead of products coming from Witt parameters.

\begin{enumerate}
\item[Case 1:]
Let $f_1=1-ax^{i_1}y^{j_1}$, $f_2=1-bx^{i_2}y^{j_2}$ and $f_3=1-cx^{i_3}y^{j_3}$. Then

\begin{align*}
&\int_T\log(f_1)\frac{df_2}{f_2}\wedge\frac{df_3}{f_3}=\\
&\int_0^1\int_0^1
\left(\int_{(0,0)}^{(\theta_1,\theta_2)}\sum_{n_1=1}^{\infty}-i_1a^{n_1}\epsilon^{n_1i_1}\epsilon^{n_1j_1}
(\exp(2\pi\sqrt{-1}n_1i_1\theta_1')d\theta_1'+\exp(2\pi\sqrt{-1}n_1i_1\theta_1')d\theta_2'\right)\times\\
&\times\sum_{n_2,n_3=1}^\infty
(i_2j_3-i_3j_2)
b^{n_2}c^{n_3}
\epsilon_1^{n_2i_2+n_3i_3}
\epsilon_2^{n_2j_2+n_3j_3}\times\\
&\times
\exp(2\pi\sqrt{-1}(n_2i_2+n_3i_3)\theta_1)
\exp(2\pi\sqrt{-1}(n_2j_2+n_3j_3)\theta_2)
d \theta_1\wedge d\theta_2=\\
=&\int_0^1\int_0^1-i_1i_2j_3
\times\\
&\times\sum_{n_1,n_2,n_3=1}^\infty\frac{1}{i_1n_1}a^{n_1}b^{n_2}c^{n_3}
\epsilon_1^{n_1i_1+n_2i_2+n_3i_3}
\epsilon_2^{n_1j_1+n_2j_2+n_3j_3}\times\\
&\times
(
\exp(2\pi\sqrt{-1}(n_1i_1+n_2i_2+n_3i_3)\theta_1)
\exp(2\pi\sqrt{-1}(n_1j_1+n_2j_2+n_3j_3)\theta_2)-\\
&
-
\exp(2\pi\sqrt{-1}(n_2i_2+n_3i_3)\theta_1)
\exp(2\pi\sqrt{-1}(n_2j_2+n_3j_3)\theta_2)
)
d \theta_1\wedge d\theta_2=\\
=&(i_2j_3-i_3j_2)
\times\\
&\times\sum_{n_1,n_2,n_3=1}^\infty\frac{1}{n_1}a^{n_1}b^{n_2}c^{n_3}
\epsilon_1^{n_1i_1+n_2i_2+n_3i_3}
\epsilon_2^{n_1j_1+n_2j_2+n_3j_3}\times\\
&\times
\int_0^1\int_0^1
\exp(2\pi\sqrt{-1}(n_2i_2+n_3i_3)\theta_1)
\exp(2\pi\sqrt{-1}(n_2j_2+n_3j_3)\theta_2)
d \theta_1\wedge d\theta_2=\\
&-(i_2j_3-i_3j_2)
\times\\
&\times\sum_{n_1,n_2,n_3=1}^\infty\frac{1}{n_1}a^{n_1}b^{n_2}c^{n_3}
\epsilon_1^{n_1i_1+n_2i_2+n_3i_3}
\epsilon_2^{n_1j_1+n_2j_2+n_3j_3}\times\\
&\times
\int_0^1\int_0^1\exp(2\pi\sqrt{-1}(n_1i_1+n_2i_2+n_3i_3)\theta_1)\times\\
&\times
\exp\left(2\pi\sqrt{-1}(n_1j_1+n_2j_2+n_3j_3)\theta_2\right)
d \theta_1\wedge d\theta_2=\\
=&-(i_2j_3-i_3j_2)
\sum^\infty_{
\begin{tabular}{cc}
$n_1,n_2,n_3=1$\\
${\bf{n}}\cdot{\bf{i}}=0$\\
${\bf{n}}\cdot{\bf{j}}=0$
\end{tabular}}
\frac{a^{n_1}b^{n_2}c^{n_3}}{n_1}\\
\end{align*}

Put
$m_k
=
\left|
\begin{tabular}{cc}
$i_{k+1}$&$i_{k+2}$\\
$j_{k+1}$&$j_{k+2}$
\end{tabular}
\right|,$
where the indices vary modulo $3$.
Let $d=gcd(m_1,m_2,m_3)$.
Then \[n_1=k\frac{|m_1|}{d},\,\,\,n_2=k\frac{|m_2|}{d},\,\,\,n_3=k\frac{|m_3|}{d}.\]

\begin{align*}
&i_2j_3\sum^\infty_{
\begin{tabular}{cc}
$n_1,n_2,n_3=1$\\
${\bf{n}}\cdot{\bf{i}}=0$\\
${\bf{n}}\cdot{\bf{j}}=0$
\end{tabular}}
\frac{a^{n_1}b^{n_2}c^{n_3}}{n_1}=\\
=&\sum^\infty_{k=1}
-i_2j_3\frac{a^{k|m_1|/d}b^{k|m_2|/d}c^{k|m_3|/d}}{k|m_1|/d}=\\
=&sign(m_1)\cdot d \cdot \sum^\infty_{k=1}
\frac{a^{k|m_1|/d}b^{k|m_2|/d}c^{kn_3}}{k}=\\
=&sign(m_1)\cdot d\cdot\log\left(1-a^{|m_1|/d}b^{|m_2|/d}c^{|m_3|/d}\right)
\end{align*}

\item[Case 2:]

Let $f_1=ax^{i_1}y^{j_1}, f_2=1-bx^{i_2}y^{j_2}, f_3=1-cx^{i_3}y^{j_3}$.

\begin{align*}
&\int\int_{T_0}\frac{df_1}{f_1}\circ\left(\frac{df_2}{f_2}\wedge\frac{df_3}{f_3}\right)=\\
=&\int_0^1\int_0^1 i_1\log(x)\times\\ 
&\times\sum_{n_2,n_3=1}^\infty
(i_2j_3-i_3j_2)
b^{n_2}c^{n_3}
x^{n_2i_2+n_3i_3}
y^{n_2j_2+n_3j_3}
\frac{dx}{x}\wedge \frac{dy}{y}
\end{align*}

If $n_2j_2+n_3j_3\neq 0$ then the integral vanishes.
Let  $n_2j_2+n_3j_3= 0$. Then \[n_2=k|j_3|/gcd(j_2,j_3)\] and
 \[n_3=k|j_2|/gcd(j_2,j_3).\] Moreover, we have a geometric series
 under the integral, namely,
\begin{align*}
g(x)=&\sum_{k=1}^\infty b^{n_2}c^{n_3}x^{n_2i_2+n_3i_3}=\\
&=
 \sum_{k=1}^\infty \left(b^{|j_3|/gcd(j_2,j_3)}c^{|j_2|/gcd(j_2,j_3)}x^{(|j_3|i_2+|j_2|i_3)/gcd(j_2,j_3)}\right)^k=\\
&=b^{|j_3|/gcd(j_2,j_3)}c^{|j_2|/gcd(j_2,j_3)}x^{sign(j_3)m_1/gcd(j_2,j_3)}\times\\
&\times
\left(1-b^{|j_3|/gcd(j_2,j_3)}c^{|j_2|/gcd(j_2,j_3)}x^{sign(j_3)m_1/gcd(j_2,j_3)}\right)^{-1}
 \end{align*}
Then \[g(x)\frac{dx}{x}=-d\log\left(1-b^{|j_3|/gcd(j_2,j_3)}c^{|j_2|/gcd(j_2,j_3)}x^{sign(j_3)m_1/gcd(j_2,j_3)}\right)\]
Let \[h(x)=1-b^{|j_3|/gcd(j_2,j_3)}c^{|j_2|/gcd(j_2,j_3)}x^{sign(j_3)m_1/gcd(j_2,j_3)}.\]
Then we have that  Case 2 is the logarithm of the 1 dimensional Contou-Carrere symbol of $x^{i_1}$ and $h(x)$
times $i_1(i_2j_3-i_3j_2)$.

Alternatively, if we sum term by term we can use the following Lemma.
\begin{lemma}
If $k\in \Z$ and $k\neq 0$ then
$\int_0^1\theta e^{2\pi\sqrt{-1}k\theta}d\theta=\frac{1}{2\pi \sqrt{-1}k}.$
\end{lemma}
\proof
\begin{align*}
\int_0^1\theta e^{2\pi\sqrt{-1}k\theta}d\theta
&=\frac{1}{2\pi \sqrt{-1}k}\int_0^1\theta d e^{2\pi\sqrt{-1}k\theta}=\\
&=\frac{1}{2\pi \sqrt{-1}k}(\theta e^{2\pi\sqrt{-1}k\theta}|_0^1-\int_0^1e^{2\pi\sqrt{-1}k\theta}d\theta)=\\
&=\frac{1}{2\pi \sqrt{-1}k}
\end{align*}

Then
\begin{align*}
&I^{1,2}(f_1,f_2,f_3)=\\
&\int_0^1\int_0^1 2\pi\sqrt{-1}i_1\theta_1\times\\ 
&\times\sum_{n_2,n_3=1}^\infty
(i_2j_3-i_3j_2)
b^{n_2}c^{n_3}
\epsilon_1^{n_2i_2+n_3i_3}
\times\\
&\times
\exp\left(2\pi\sqrt{-1}(k|j_3|i_2/d+k|j_2|i_3/d)\theta_1\right)
d \theta_1\wedge d\theta_2=\\
=&\sum_{k_1=1}^\infty
2\pi\sqrt{-1}i_1(i_2j_3-i_3j_2)
\frac{\left(b^{|j_3|/d}c^{|j_2|/d}\epsilon_1^{(|j_3|i_2/d+|j_2|i_3/d})\right)^{k_1}}{2\pi \sqrt{-1}k_1|j_3|i_2/d+k_1|j_2|i_3/d}+\\
=&sign(j_3)\cdot i_1\cdot d \cdot
\sum_{k_1=1}^\infty
\frac{\left(b^{|j_3|/d}c^{|j_2|/d}\epsilon_1^{(|j_3|i_2/d+|j_2|i_3/d})\right)^{k_1}}{k_1}=\\
=&-sign(j_3)\cdot i_1\cdot d \cdot \log \left(1-b^{|j_3|/d}c^{|j_2|/d}\epsilon_1^{sign(j_3)m_1/d}\right)
\end{align*}

Then
\begin{align*}
&I^{2,1}(f_1,f_2,f_3)=\\
&\int_0^1\int_0^1 2\pi\sqrt{-1}j_1\theta_2\times\\ 
&\times\sum_{n_2,n_3=1}^\infty
(i_2j_3-i_3j_2)
b^{n_2}c^{n_3}
\epsilon_1^{n_2i_2+n_3i_3}
\times\\
&\times
\exp\left(2\pi\sqrt{-1}(n_2j_2+n_3j_3)\theta_2\right)
d \theta_1\wedge d\theta_2=\\
=&\sum_{k_2=1}^\infty
2\pi\sqrt{-1}j_2(i_2j_3-i_3j_2)
\frac{\left(b^{|i_3|/d}c^{|i_2|/d}\epsilon_1^{(|i_3|j_2/d+|i_2|j_3/d})\right)^{k_2}}{2\pi \sqrt{-1}k_2|i_3|j_2/d+k_2|i_2|j_3/d}+\\
=&sign(i_3)\cdot i_1\cdot d \cdot
\sum_{k_2=1}^\infty
\frac{\left(b^{|i_3|/d}c^{|i_2|/d}\epsilon_2^{(|i_3|j_2/d+|i_2|j_3/d})\right)^{k_2}}{k_2}=\\
=
&-sign(i_3)\cdot j_1\cdot d \cdot \log \left(1-b^{|i_3|/d}c^{|i_2|/d}\epsilon_2^{sign(i_3)m_1/d}\right)
\end{align*}

\item[Case 3:] Let $f_1=1-ax^{i_1}y^{j_1}, f_2=bx^{i_2}y^{j_2}, f_3=1-cx^{i_3}y^{j_3}$.

\begin{align*}
&\int\int_{T_0}\frac{df_1}{f_1}\circ\left(\frac{df_2}{f_2}\wedge\frac{df_3}{f_3}\right)=\\
=&\int_0^1\int_0^1
\left(\int_0^{\theta_1}\sum_{n_1=1}^{\infty}-i_1a^{n_1}\epsilon_1^{n_1i_1}\epsilon_2^{n_1j_1}
\exp\left(2\pi\sqrt{-1}n_1(i_1\theta_1'+j_1\theta_2\right)(d\theta_1'\right)\times\\
&\times(2\pi \sqrt{-1})(i_2d\theta_1+j_2d\theta_2)\wedge\\
&\wedge\sum_{n_3=1}^\infty
c^{n_3}
\epsilon_1^{n_3i_3}
\epsilon_2^{n_3j_3}
\exp\left(2\pi\sqrt{-1}(n_3i_3)\theta_1\right)
\exp\left(2\pi\sqrt{-1}(n_3j_3)\theta_2\right)
(i_3d\theta_1+j_3d\theta_2)=\\
=&\sum_{n_1,n_3=1}^\infty
\frac{i_1(i_2j_3-i_3j_2)}{n_1i_1}a^{n_1}c^{n_3}\epsilon_1^{n_1i_1+n_3i_3}\epsilon_2^{n_1j_1+n_3j_3}\times\\
&\times\int_0^1\int_0^1
\exp\left(2\pi\sqrt{-1}(n_1i_1+n_3i_3)\theta_1\right)
\exp\left(2\pi\sqrt{-1}(n_1j_1+n_3j_3)\theta_2\right)d\theta_1\wedge d\theta_2=\\
\end{align*}
The last double integral vanishes if $n_1i_1+n_3i_3\neq 0$ or if $n_1j_1+n_3j_3\neq 0$. If both
 $n_1i_1+n_3i_3= 0$ and $n_1j_1+n_3j_3=0$ then the two vectors $(i_1,i_3)$
 and $(j_1,j_3)$ are linearly dependent. Let 
 $n_1=k|i_3|/(i_1,i_3)$ and $n_3=k|i_1|/(i_1,i_3)$.

Note that $i_2j_3-i_3j_2=0$. 
Then the contribution from Case 3: becomes
\begin{align*}
I_3=&\sum_{k=1}^{\infty}\frac{i_2j_3}{n_1}a^{n_1}c^{n_3}=\\
=&\sum_{k_1=1}^\infty\frac{i_2j_3(j_1,j_3)}{|j_3|k_1}\left(a^{|j_3|/(j_1,j_3)}c^{|j_1|/(j_1,j_3)}\right)^{k_1}-\\
&-\sum_{k_2=1}^\infty\frac{i_3j_2(i_1,i_3)}{|i_3|k_2}\left(a^{|i_3|/(i_1,i_3)}c^{|i_1|/(i_1,i_3)}\right)^{k_2}=\\
=&-sign(j_3)j_2(j_1,j_3)\log\left(1-a^{|j_3|/(j_1,j_3)}c^{|j_1|/(j_1,j_3)}\right)+\\
&+sign(i_3)j_2(i_1,i_3)\log\left(1-a^{|i_3|/(i_1,i_3)}c^{|i_1|/(i_1,i_3)}\right)
\end{align*}

\item[Case 4:] Let $f_1=1-ax^{i_1}y^{j_1}, f_2=1-bx^{i_2}y^{j_2}, f_3=cx^{i_3}y^{j_3}$.

The contribution from Case 4 is similar to Case 3.
\begin{align*}
I_4=&\sum_{k=1}^{\infty}\frac{i_2j_3-i_3j_2}{n_1}a^{n_1}b^{n_2}&=\\
=&\sum_{k_1=1}^\infty\frac{i_2j_3(i_1,i_2)}{|i_2|k_1}\left(a^{|i_2|/(i_1,i_2)}b^{|i_1|/(i_1,i_2)}\right)^{k_1}-\\
&-\sum_{k_2=1}^\infty\frac{i_3j_2(j_1,j_2)}{|j_2|k_2}\left(a^{|j_2|/(j_1,j_2)}b^{|j_1|/(j_1,j_2)}\right)^{k_2}=
\\
=&-sign(i_2)j_3(i_1,i_2)\log\left(1-a^{|i_2|/(i_1,i_2)}b^{|i_1|/(i_1,i_2)}\right)
+\\
&+
sign(j_2)i_3(j_1,j_2)\log\left(1-a^{|j_2|/(j_1,j_2)}b^{|j_1|/(j_1,j_2)}\right)
\end{align*}

\item[Case 5:]

Let $f_1=x^{i_1}y^{j_1}, f_2=x^{i_2}y^{j_2}, f_3=C(1-cx^{i_3}y^{j_3})$.

\begin{align*}
&\int\int_{T_0}\frac{df_1}{f_1}\circ\left(\frac{df_2}{f_2}\wedge\frac{df_3}{f_3}\right)=\\
=&\int_0^1\int_0^1
\left(
\int_{(0,0)}^{(\theta_1,\theta_2)}(i_1d\theta_1'+j_1d\theta_2'\right)
\times\\
&\times(2\pi \sqrt{-1})(i_2d\theta_1+j_2d\theta_2)\wedge\\
&\wedge\sum_{n_3=1}^\infty
-c^{n_3}
\epsilon_1^{n_3i_3}
\epsilon_2^{n_3j_3}
\exp\left(2\pi\sqrt{-1}(n_3i_3)\theta_1\right)
\exp\left(2\pi\sqrt{-1}(n_3j_3)\theta_2\right)
(i_3d\theta_1+j_3d\theta_2)=\\
=&\int_0^1\int_0^1
m_1i_1\theta_1\times\\
&\times
\sum_{n_3=1}^\infty
c^{n_3}
\epsilon_1^{n_3i_3}
\epsilon_2^{n_3j_3}
\exp\left(2\pi\sqrt{-1}(n_3i_3)\theta_1\right)
\exp\left(2\pi\sqrt{-1}(n_3j_3)\theta_2\right)
d \theta_1\wedge d\theta_2+\\
&+\int_0^1\int_0^1
j_1m_1\theta_2\times\\
&\times
\sum_{n_3=1}^\infty
c^{n_3}
\epsilon_1^{n_3i_3}
\epsilon_2^{n_3j_3}
\exp\left(2\pi\sqrt{-1}(n_3i_3)\theta_1\right)
\exp\left(2\pi\sqrt{-1}(n_3j_3)\theta_2\right)
d \theta_1\wedge d\theta_2
\end{align*}

The last integral is different from zero only if $i_3=0$. In that case, we have

\begin{align*}
&\int\int_{T_0}\frac{df_1}{f_1}\circ\left(\frac{df_2}{f_2}\wedge\frac{df_3}{f_3}\right)=\\
=&\int_0^1\int_0^1
i_1\theta_1(-i_3j_2)\times\\
&\times\sum_{n_3=1}^\infty
c^{n_3}
\epsilon_1^{n_3i_3}
\epsilon_2^{n_3j_3}
\exp\left(2\pi\sqrt{-1}(n_3i_3)\theta_1\right)
\exp\left(2\pi\sqrt{-1}(n_3j_3)\theta_2\right)
d \theta_1\wedge d\theta_2+\\
&+
\int_0^1\int_0^1
j_1\theta_2(i_2j_3)\times\\
&\times\sum_{n_3=1}^\infty
c^{n_3}
\epsilon_1^{n_3i_3}
\epsilon_2^{n_3j_3}
\exp\left(2\pi\sqrt{-1}(n_3i_3)\theta_1\right)
\exp\left(2\pi\sqrt{-1}(n_3j_3)\theta_2\right)
d \theta_1\wedge d\theta_2=\\
=&\sum_{n_3=1}^\infty
\frac{-i_1i_3j_2}{n_3i_3}c^{n_3}\epsilon_1^{n_3i_3}+\\
&+\sum_{n_3=1}^\infty
\frac{i_1i_2j_3}{n_3i_3}c^{n_3}\epsilon_2^{n_3j_3}=\\
=&-\frac{i_1i_3j_2}{i_3}\sum_{n_3=1}^{\infty}\frac{(c\epsilon_1^{i_3})^{n_3}}{n_3}
+
\frac{i_1i_2j_3}{j_3}\sum_{n_3=1}^{\infty}\frac{(c\epsilon_1^{i_3})^{n_3}}{n_3}
=\\
=&i_1j_2\log(1-c\epsilon_1^{i_3})-i_2j_1\log(1-c\epsilon_2^{i_3})
\end{align*}

Thus, \[I_5=i_1j_2\log(1-c(-P_1)^{i_3})-i_2j_1\log(1-c(-P_2)^{j_3})\]

\item[6:]

Let $f_1=x^{i_1}y^{j_1}, f_2=B(1-bx^{i_2}y^{j_2}), f_3=x^{i_3}y^{j_3}$.
Similarly to case 5, we obtain

 \[I_6=i_3j_1\log(1-b(-P_1)^{i_2})-i_1j_3\log(1-b(-P_2)^{j_2})\]

\item[7:]

Let $f_1=A(1-ax^{i_1}y^{j_1}), f_2=x^{i_2}y^{j_2}, f_3=x^{i_3}y^{j_3}$.

\begin{align*}
&\int\int_{T_0}\frac{df_1}{f_1}\circ\left(\frac{df_2}{f_2}\wedge\frac{df_3}{f_3}\right)=\\
=&\int_0^1\int_0^1
\left(\int_0^{\theta_1}\sum_{n_1=1}^{\infty}-i_1a^{n_1}\epsilon_1^{n_1i_1}\epsilon_2^{n_1j_1}
\exp(2\pi\sqrt{-1}n_1i_1\theta_1')
\exp(2\pi\sqrt{-1}n_1j_1\theta_2)
d\theta_1'\right)\times\\
&\times i_2j_31d\theta_1\wedge d\theta_2=\\
=&\int_0^1\int_0^1\sum_{n_1=1}^\infty -\frac{i_1i_2j_3a^{n_1}\epsilon_1^{n_1i_1}\epsilon_2^{n_1j_1}}{n_1i_1}\times\\
&\times
\left(\exp\left(2\pi\sqrt{-1}(n_1i_1)\theta_1\right)-1\right)\exp(2\pi\sqrt{-1}n_1j_1\theta_2)
d \theta_1\wedge d\theta_2
\end{align*}
The integral is zero if $j_1\neq 0$. If $j_2=0$ then
we obtain
\begin{align*}
&\int\int_{T_0}\frac{df_1}{f_1}\circ\left(\frac{df_2}{f_2}\wedge\frac{df_3}{f_3}\right)=\\
=&-\int_0^1
\sum_{n_1=1}^\infty -\frac{i_1i_2j_3a^{n_1}\epsilon_1^{n_1i_1}}{n_1i_1}
\exp\left(2\pi\sqrt{-1}(n_1i_1)\theta_1\right)d \theta_1-\\
&\sum_{n_1=1}^\infty -\frac{i_1i_2j_3a^{n_1}\epsilon_1^{n_1i_1}}{n_1i_1}=
0-i_2j_3\log(1-a\epsilon_1^{i_1})
\end{align*}

Note that the last integral vanishes if $i_1\neq 0$. However, if $i_1=0$ and $j_1=0$ then $f_1=1$ and the integral vanishes again.
 Thus,
 
\[I_7=-m_1\log(1-aR_1^{i_1})\]

\item[8:] Let $f_1=x^{i_1}y^{j_1}, f_2=x^{i_2}y^{j_2}, f_3=x^{i_3}y^{j_3}$.

For this case it is better to use the notation $I^{1,2}(f_1,f_2,f_3)$ and $I^{2,1}(f_1,f_2,f_3)$ from Subsection 1.3. Using that the logarithm of the Parshin symbol (see \cite{rec2}) can be written
as 
\begin{align*}
&I^{1,2}(f_1,f_2,f_3)
+I^{1,2}(f_3,f_1,f_2)
+I^{1,2}(f_2,f_3,f_1)-\\
&-I^{2,1}(f_1,f_2,f_3)
-I^{2,1}(f_3,f_1,f_2)
-I^{2,1}(f_1,f_2,f_3,f_3)
\end{align*}

The iteration $\frac{df_1}{f_1}\circ\left(\frac{df_2}{f_2}\wedge\frac{df_3}{f_3}\right)$ gives  $4$ from the above $6$ terms, namely
\[\frac{(2\pi i)^3}{2}(i_1i_2j_3+i_3i_1j_3-i_3j_1j_2-i_2j_3j_1)\]

The ones that are not present are monomials $i_2i_3j_1$ and $j_2j_3i_1$.

\end{enumerate}

\begin{definition} (Contou-Carrere symbol for surfaces)
Let 
\[f_1=a_1x^{\nu_1(f_1)}y^{\nu_2(f_1)}\prod_{i_1>-N}\prod_{j_1>-N_{i_1}}(1-a_{i_1,j_1}x^{i_1}y^{j_1}),\]
\[f_2=a_2x^{\nu_1(f_2)}y^{\nu_2(f_2)}\prod_{i_2>-N}\prod_{j_2>-N_{i_2}}(1-a_{i_2,j_2}x^{i_2}y^{j_2}),\]
\[f_3=a_3x^{\nu_1(f_3)}y^{\nu_2(f_3)}\prod_{i_3>-N}\prod_{j_3>-N_{i_3}}(1-a_{i_3,j_3}x^{i_3}y^{j_3}).\]
Let 
\[T(f_1,f_2,f_3)=\prod_{i_1,i_2,i_3,j_1,j_2,j_3}\left(1-a_{i_1,j_1}^{|m_1|/d}a_{i_2,j_2}^{|m_2|/d}a_{i_3,j_3}^{|m_3|/d}\right)^{sign(m_1)d},\]
where $m_k=i_{k+1}j_{k+2}-i_{k+2}j_{k+1}$ for $k$ modulo $3$ and $d$ is the greatest common divisor of $m_1,m_2,m_3$.
\[Q_1(f_1,f_2,f_3,P_1)=\prod_{\nu_1(f_1),i_2,i_3,j_2,j_3}
\left(1-a_{i_2,j_2}^{|j_3|/d}a_{i_3,j_3}^{|j_2|/(j_2,j_3)}P_1^{sign(j_3)m_1/(j_2,j_3)}\right)
^{-sign(j_3) \nu_1(f_1) (j_2,j_3)}\]
\[Q_2(f_1,f_2,f_3,P_2)=\prod_{\nu_2(f_1),i_2,i_3,j_2,j_3}
\left(1-a_{i_2,j_2}^{|i_3|/(i_2,i_3)}a_{i_3,j_3}^{|i_2|/(i_2,i_3)}P_2^{sign(i_3)m_1/d}\right)
^{sign(i_3) \nu_2(f_1) (i_2,i_3)}\]
\[Q_3(f_1,f_2,f_3)=\prod_{\nu_1(f_2),i_1,i_3,j_1,j_3} 
\left(1-a_{i_1,j_1}^{|i_3|/(i_1,i_3)}a_{i_3,j_3}^{|i_1|/(i_1,i_3)}\right)^{sign(i_3)\nu_2(f_2)(i_1,i_3)}\]
\[Q_4(f_1,f_2,f_3)=\prod_{\nu_2(f_2),i_1,i_3,j_1,j_3} 
\left(1-a_{i_1,j_1}^{|j_3|/(j_1,j_3)}a_{i_3,j_3}^{|j_1|/(j_1,j_3)}\right)^{-sign(j_3)\nu_1(f_2)(j_1,j_3)}\]
\[R_1(f_1,f_2,f_3,P_1)=\prod_{\nu_1(f_2),\nu_2(f_3),\nu_1(f_3),\nu_2(f_2),i_1}
(1-a_{i_1,0}P_1^{i_1})^{\nu_1(f_2)\nu_2(f_3)-\nu_1(f_3)\nu_2(f_2)}\]
\[R_2(f_1,f_2,f_3,P_2)=\prod_{\nu_2(f_2),\nu_2(f_3),\nu_1(f_3),\nu_2(f_2),j_1}
(1-a_{0,j_1}P_2^{j_1})^{\nu_1(f_2)\nu_2(f_3)-\nu_1(f_3)\nu_2(f_2)}\]
\[S(f_1,f_2,f_3)=(-1)^A, \mbox{ where }A=i_1i_2j_3+i_2i_3j_1+i_3i_1j_2-i_1j_2j_3-i_2j_3j_1-i_3j_1j_2\]
Then the Contou-Carrere symbol is a formal product
\begin{align*}
\{f_1,f_2,f_3\}_{C,0}^P =&T(f_1,f_2,f_3)S(f_1,f_2,f_3)\times\\
&
\times\prod_{cyclic}Q_1(f_1,f_2,f_3,P_1)Q_2(f_1,f_2,f_3,P_2)Q_3(f_1,f_2,f_3)Q_4(f_1,f_2,f_3)\times\\
&
\times\prod_{cyclic} R_1(f_1,f_2,f_3,P_1)R_2(f_1,f_2,f_3,P_2),
\end{align*}
where $\prod_{cyclic}$ is a product over a cyclic permutation of the order of the functions $f_1,f_2,f_3$.
\end{definition}

\subsection{Semi-local formulas}

Let $f_1=1-a_1x^{i_1}y^{j_1}$,
$f_2=1-a_2x^{i_2}y^{j_2}$,
$f_1=1-a_3x^{i_3}y^{j_3}-a_4x^{i_4}y^{j_4}$,

\begin{align*}
&\int_{T}\frac{df_1}{f_1}\circ\left(\frac{df_2}{f_2}\wedge\frac{df_3}{f_3}\right)=\\
&
=
\left|\begin{tabular}{cc}
$i_2$&$i_3$\\
$j_2$&$j_3$
\end{tabular}
\right|\int_T\sum_{n_1,n_2,n_3=1}^\infty\sum_{n_4=0}^\infty
-\frac{2}{n_1}\frac{(n_3+n_4)!}{n_3!n_4!}\prod_{k=1}^4a_k^{n_k}x^{i_kn_k}y^{j_kn_k}
\frac{dx}{x}\wedge\frac{dy}{y}=\\
&
-
\left|\begin{tabular}{cc}
$i_2$&$i_4$\\
$j_2$&$j_4$
\end{tabular}
\right|\int_T\sum_{n_1,n_2,n_4=1}^\infty\sum_{n_3=0}^\infty
\frac{2}{n_1}\frac{(n_3+n_4)!}{n_3!n_4!}\prod_{k=1}^4a_k^{n_k}x^{i_kn_k}y^{j_kn_k}
\frac{dx}{x}\wedge\frac{dy}{y}
\end{align*}

\begin{align*}
\left|\begin{tabular}{cc}
$i_2$&$i_3$\\
$j_2$&$j_3$
\end{tabular}
\right|\int_T\sum_{n_1,n_2,n_3=1}^\infty\sum_{n_4=0}^\infty
-\frac{2}{n_1}\frac{(n_3+n_4)!}{n_3!n_4!}\prod_{k=1}^4a_k^{n_k}x^{i_kn_k}y^{j_kn_k}
\frac{dx}{x}\wedge\frac{dy}{y}=\\
\end{align*}

Let $\left|\begin{tabular}{cc}
$i_2$&$i_3$\\
$j_2$&$j_3$
\end{tabular}
\right|\neq 0$. Then by row reduction we obtain the following

\begin{align*}
&\left[\begin{tabular}{ccccc}
$i_1$ &$i_2$ &$i_3$ &$i_4$\\
$j_1$ &$j_2$ &$j_3$ &$j_4$
\end{tabular}\right]
\rightarrow
\left[\begin{tabular}{ccccc}
$\left|\begin{tabular}{cc}
$i_1$&$i_3$\\
$j_1$&$j_3$
\end{tabular}
\right|$ &
$\left|\begin{tabular}{cc}
$i_2$&$i_3$\\
$j_2$&$j_3$
\end{tabular}
\right|$ &
$0$ &
$\left|\begin{tabular}{cc}
$i_4$&$i_3$\\
$j_4$&$j_3$
\end{tabular}
\right|$\\
\\
$\left|\begin{tabular}{cc}
$i_2$&$i_1$\\
$j_2$&$j_1$
\end{tabular}
\right|$ &
$0$&
$\left|\begin{tabular}{cc}
$i_2$&$i_3$\\
$j_2$&$j_3$
\end{tabular}
\right|$ &
$\left|\begin{tabular}{cc}
$i_2$&$i_4$\\
$j_2$&$j_4$
\end{tabular}
\right|$ 
\end{tabular}\right]
\end{align*}

Then the sums have to vanish, $\sum{k=1}^4i_kn_k=\sum_{k=1}^4 j_kn_k=0$.
Using the above row reduction, we obtain 
\begin{align*}
&n_1=
\left|\begin{tabular}{cc}
$i_2$&$i_3$\\
$j_2$&$j_3$
\end{tabular}
\right|k_1/d\\
&n_2=\left(-\left|\begin{tabular}{cc}
$i_4$&$i_3$\\
$j_4$&$j_3$
\end{tabular}
\right|k_4
-\left|\begin{tabular}{cc}
$i_2$&$i_3$\\
$j_2$&$j_3$
\end{tabular}
\right|k_1\right)/d\\
&n_3=\left(-\left|\begin{tabular}{cc}
$i_2$&$i_4$\\
$j_2$&$j_4$
\end{tabular}
\right|k_4
-\left|\begin{tabular}{cc}
$i_2$&$i_1$\\
$j_2$&$j_1$
\end{tabular}
\right|k_1\right)/d\\
&n_4=\left|\begin{tabular}{cc}
$i_2$&$i_3$\\
$j_2$&$j_3$
\end{tabular}
\right|k_4/d,
\end{align*}
where $d$ is the greatest common divisor of 
$\left|\begin{tabular}{cc}
$i_2$&$i_3$\\
$j_2$&$j_3$
\end{tabular}
\right|,$
$\left|\begin{tabular}{cc}
$i_4$&$i_3$\\
$j_4$&$j_3$
\end{tabular}
\right|$,
$\left|\begin{tabular}{cc}
$i_2$&$i_3$\\
$j_2$&$j_3$
\end{tabular}
\right|$, 
$\left|\begin{tabular}{cc}
$i_2$&$i_4$\\
$j_2$&$j_4$
\end{tabular}
\right|$ and
$\left|\begin{tabular}{cc}
$i_2$&$i_1$\\
$j_2$&$j_1$
\end{tabular}\right|$.

Let \[M=\left|\begin{tabular}{cc}
$i_2$&$i_3$\\
$j_2$&$j_3$
\end{tabular}
\right|/d-\left|\begin{tabular}{cc}
$i_2$&$i_4$\\
$j_2$&$j_4$
\end{tabular}
\right|/d\] and 
\[N=\left|\begin{tabular}{cc}
$i_2$&$i_1$\\
$j_2$&$j_1$
\end{tabular}
\right|/d.\] Let also $\xi_M$ and $\xi_N$ be  a primitive $M$-th and $N$-th root of unity.

\begin{align*}
&\left|\begin{tabular}{cc}
$i_2$&$i_3$\\
$j_2$&$j_3$
\end{tabular}
\right|\int_T\sum_{n_1,n_2,n_3=1}^\infty\sum_{n_4=0}^\infty
-\frac{2}{n_1}\frac{(n_3+n_4)!}{n_3!n_4!}\prod_{k=1}^4a_k^{n_k}x^{i_kn_k}y^{j_kn_k}
\frac{dx}{x}\wedge\frac{dy}{y}=\\
&
=-d
\left|\begin{tabular}{cc}
$i_2$&$i_3$\\
$j_2$&$j_3$
\end{tabular}
\right|\sum_{k_1=1}^\infty\sum_{k_4=0}^\infty\frac{1}{k_1
\left|\begin{tabular}{cc}
$i_2$&$i_3$\\
$j_2$&$j_3$
\end{tabular}
\right|}a_1^{n_1}a_2^{n_2}a_3^{n_3}a_4^{n_4}\frac{(n_3+n_4)!}{n_3!n_4!}=\\
&=-d\sum_{k_1=1}^\infty\sum_{k_4=0}^\infty\frac{1}{k_1}
a_1^{n_1}a_2^{n_2}(a_3+a_4)^{n_3+n_4}=\\
&=-d\sum_{k_1=1}^\infty\sum_{k_4=0}^\infty\frac{1}{k_1}
a_1^{n_1}a_2^{n_2}(a_3+a_4)^{n_3+n_4}=\\
&=-d\sum_{k_1=1}^\infty\sum_{k_4=0}^\infty\frac{1}{k_1}
a_1^{\left|\begin{tabular}{cc}
$i_2$&$i_3$\\
$j_2$&$j_3$
\end{tabular}
\right|k_1/d}
a_2^{
\left(-\left|\begin{tabular}{cc}
$i_4$&$i_3$\\
$j_4$&$j_3$
\end{tabular}
\right|k_4
-\left|\begin{tabular}{cc}
$i_2$&$i_3$\\
$j_2$&$j_3$
\end{tabular}
\right|k_1\right)/d
}\times\\
&\times
\sum_{m_1,m_2=1}^M(\xi_M^{m_1}a_3+\xi_M^{m_2}a_4))^{
\left(\left|\begin{tabular}{cc}
$i_2$&$i_3$\\
$j_2$&$j_3$
\end{tabular}
\right|k_4
-\left|\begin{tabular}{cc}
$i_2$&$i_4$\\
$j_2$&$j_4$
\end{tabular}
\right|k_4
\right)/d}\times\\
&\times
\sum_{n=1}^N(\xi_N^{n_1}a_3+\xi_N^{n_2}a_4))^{
-\left|\begin{tabular}{cc}
$i_2$&$i_1$\\
$j_2$&$j_1$
\end{tabular}
\right|k_1/d}=\\
\\
&=
\sum_{m_1,m_2=1}^M
\left(1-a_2^{
\left(-\left|\begin{tabular}{cc}
$i_4$&$i_3$\\
$j_4$&$j_3$
\end{tabular}
\right|\right)/d}
(\xi_M^{m_1}a_3+\xi_M^{m_2}a_4)^{
\left(\left|\begin{tabular}{cc}
$i_2$&$i_3$\\
$j_2$&$j_3$
\end{tabular}
\right|
-\left|\begin{tabular}{cc}
$i_2$&$i_4$\\
$j_2$&$j_4$
\end{tabular}
\right|\right)/d}\right)^{-1}\times\\
&d\times
\sum_{n_1,n_2=1}^N
\log\left(
1-a_1^{
\left|\begin{tabular}{cc}
$i_2$&$i_3$\\
$j_2$&$j_3$
\end{tabular}
\right|/d
}\right.
a_2^{
-\left|\begin{tabular}{cc}
$i_2$&$i_3$\\
$j_2$&$j_3$
\end{tabular}
\right|/d
}
\left.(\xi_N^{n_1}a_3+\xi_N^{n_2}a_4)
^
{-\left|
\begin{tabular}{cc}
$i_2$&$i_1$\\
$j_2$&$j_1$
\end{tabular}
\right|
/d}
\right)
\end{align*}

\renewcommand{\em}{\textrm}

\begin{small}

\renewcommand{\refname}{ {\flushleft\normalsize\bf{References}} }
    
\end{small}
Ivan E. Horozov\\
\begin{small}
Washington University in St Louis\\
Department of Mathematics\\
Campus Box 1146\\
One Brookings Drive\\
St Louis, MO 63130\\
USA\\
\end{small}
\\
Zhenbin Luo
\begin{small}
\end{small}

\end{document}